\documentclass[11pt]{article}
\usepackage[T1]{fontenc}
\usepackage{lmodern}
\usepackage{amsmath,amssymb,amsthm,mathrsfs}
\usepackage[margin=0.94in]{geometry}
\usepackage[colorlinks=true,linkcolor=blue,citecolor=blue,urlcolor=blue]{hyperref}
\newtheorem{theorem}{Theorem}[section]
\newtheorem{proposition}[theorem]{Proposition}
\newtheorem{lemma}[theorem]{Lemma}
\newtheorem{corollary}[theorem]{Corollary}
\theoremstyle{definition}
\newtheorem{definition}[theorem]{Definition}
\theoremstyle{remark}
\newtheorem{remark}[theorem]{Remark}
\newcommand{\OC}{\mathcal O_C}
\newcommand{\OEv}{\mathcal O_E}
\newcommand{\Spa}{\operatorname{Spa}}

\newcommand{\Sh}{\operatorname{Sh}}

\newcommand{\an}{\mathrm{an}}
\newcommand{\HT}{\mathrm{HT}}
\newcommand{\Fl}{\mathscr F\!\ell}
\newcommand{\colim}{\mathop{\mathrm{colim}}}
\newcommand{\ilim}{\mathop{\mathrm{lim}}}
\newcommand{\eps}{\varepsilon}

\newcommand{\cS}{\mathscr S}
\newcommand{\cH}{\mathcal H}
\title{PERFECTOIDNESS OF PROPER SHIMURA VARIETIES\texorpdfstring{\\}{ }WITH LIMPID INTEGRAL MODELS}
\author{Ali Partofard}
\date{}
\begin{document}
\maketitle

\begin{abstract}
We prove that any proper Shimura variety admitting a limpid integral model in the sense of Madapusi Pera becomes a perfectoid space at infinite level at $p$. Our method provides a new geometric approach that circumvents the need for a universal abelian scheme, making it applicable beyond Shimura varieties of Hodge type. Inspired by Scholze's proof of the perfectoidness of the Siegel modular variety, we construct analogues of the canonical subgroup and the anticanonical tower within a strict neighborhood of the ordinary locus. To achieve this we utilize the geometry of the universal $G$-aperture over the limpid integral model. Because every Shimura variety admits a limpid integral model for sufficiently large primes, our result implies that all proper Shimura varieties are globally perfectoid at infinite level for big enough primes.
\end{abstract}

\section{Introduction}

Scholze \cite{Scholze2015} proved that the Siegel moduli space of abelian varieties, and more generally any Shimura variety of Hodge type, becomes a perfectoid space at infinite level at $p$. A crucial ingredient in Scholze's strategy is the existence of a universal abelian variety (and its associated $p$-divisible group), which allows for the explicit construction of the Hodge-Tate period map and the anticanonical tower. 

For general Shimura varieties—particularly those not of Hodge type—the absence of a universal abelian scheme poses a significant obstacle to establishing global perfectoidness at infinite level. Despite this obstruction on the global level, there has been major progress regarding the local geometry. Most notably, the recent work of T. He \cite{he2026perfectoidness} proves that the completed local rings (equivalently, the completed stalks) of general Shimura varieties at infinite level are perfectoid spaces.

Recent advances in the construction of integral models, specifically the theory of limpid integral models in the sense of Madapusi Pera ~\cite[Definition 1.2.2]{MY2026}, offer new tools to control the $p$-adic integral geometry of these spaces. The purpose of this article is to prove that the existence of such a limpid integral model provides exactly the right structure to deduce global perfectoidness at infinite level.

\begin{theorem}\label{thm:main-intro}
Let $E/\mathbb Q_p$ be the local reflex field, assumed unramified,
let $K^p$ be neat, and suppose that
$S=\Sh_{K^pK,E}^{\an}$ has a smooth proper limpid integral canonical
model $\cS/\OEv$ which is a scheme. For any complete algebraically closed extension $C/E$,
the full $p$-level diamond
\[
 \ilim_{K'\subset K}(\Sh_{K^pK',C}^{\an})^\diamond
\]
is represented by a perfectoid space over $C$.
\end{theorem}

Crucially, every Shimura variety admits a limpid integral model for big enough primes \cite{MY2026}. This observation immediately elevates our main theorem into a completely general result for sufficiently large $p$:

\begin{corollary}
\label{cor:large_primes}
For any proper Shimura datum $(G, X)$, there exists a constant $C > 0$ such that for all primes $p > C$, the infinite-level Shimura variety $\mathcal{S}_{K^p, \infty}$ is a perfectoid space. 
\end{corollary}

Our proof closely follows the strategy established by Scholze for the Siegel modular variety $\mathcal{A}_g$, but we should substitute the role of the universal abelian variety with the \textit{universal $G$-aperture}.

\subsection{Proof Strategy}

\medskip\noindent
\textbf{Canonical Lifts and the $p$-adic Hecke Correspondence.} 
Our starting point is a recent result of Madapusi and Youcis, who demonstrated the existence of a canonical lift of the Frobenius isogeny over the $\mu$-ordinary locus within the integral models of Shimura varieties. From a geometric perspective, this canonical Frobenius lift naturally equips us with a section of the $p$-adic Hecke correspondence restricted to the $\mu$-ordinary locus.

\medskip\noindent
\textbf{The Generalized Hesse Invariant and Canonical Subgroups.} 
To move beyond the strict $\mu$-ordinary locus, we utilize the generalized Hesse invariant. By carefully bounding the valuations of this invariant, we are able to analytically extend the aforementioned section of the Hecke correspondence to an open neighborhood of the ordinary locus in the generic fiber. Crucially, this extended section serves as the precise group-theoretic analogue of the classical canonical subgroup for abelian varieties. 

\medskip\noindent
\textbf{The Anticanonical Tower and Local Perfectoidness.} 
Equipped with this generalized canonical subgroup, we proceed to iteratively construct the anticanonical tower over the generic fiber of the Shimura variety. By analyzing the dynamics of the transition maps in this tower, we show that its inverse limit is a perfectoid space. We then invoke the almost purity theorem to transfer this structure back to the standard tower of Shimura varieties. This allows us to deduce that the actual limit of the Shimura variety at level $p^\infty$ is perfectoid over our open neighborhood of the ordinary locus.

\medskip\noindent
\textbf{Global Perfectoidness via the Hecke Action.} 
In the final step, we upgrade this local perfectoidness to the entire generic fiber. We achieve this by exploiting the natural Hecke action on the infinite-level Shimura variety in conjunction with the geometry of the associated flag variety. By translating our perfectoid open neighborhood via Hecke operators and analyzing the image under the Hodge--Tate period map, we sweep out the entire space, thereby establishing that the whole Shimura variety of level $p^\infty$ is globally a perfectoid space.
\section{Review of $G$-Apertures and the Ordinary Locus}

In this section, we review the theory of prismatic $F$-gauges and $G$-apertures developed by Madapusi and Youcis \cite{MY2026}, which serves as our primary replacement for the universal abelian scheme. These objects provide a purely group-theoretic approach to $p$-divisible groups with $G$-structure.

\subsection{Prismatic $F$-gauges and $(G,\mu)$-Apertures}

We begin by recalling the definition of prismatic $F$-gauges. Let $\mathfrak{X}$ be a derived $p$-adic formal scheme, and let $\mathfrak{X}^{\mathrm{syn}}$ denote its syntomification, a derived $p$-adic formal stack canonically associated with $\mathfrak{X}$\cite[Definition 6.1.1]{bhatt2022prismatic}. 

\begin{definition}[Prismatic $F$-gauges]
A vector bundle in prismatic $F$-gauges, or simply a \textbf{vector bundle $F$-gauge} over $\mathfrak{X}$, is a vector bundle over the formal stack $\mathfrak{X}^{\mathrm{syn}}$. 
\end{definition}

To define $G$-apertures, we fix the following setup.
\paragraph{Setup.} Let $G$ be a smooth connected affine group scheme over $\mathbb{Z}_p$, and let $\breve{\mathcal{O}}$ be the ring of integers in a finite unramified extension of $\mathbb{Q}_p$. Let $\mu \colon \mathbb{G}_m \to G_{\breve{\mathcal{O}}}$ be a $1$-bounded cocharacter (i.e., its weights on $\mathrm{Lie}(G)_{\breve{\mathcal{O}}}$ via the adjoint action are bounded above by 1). This $\mu$ defines a canonical $G$-torsor $Q_\mu$ over the classifying stack $B\mathbb{G}_{m, \breve{\mathcal{O}}}$.

\begin{definition}[$(G, \mu)$-apertures]
For a $p$-complete ring $R \in \mathrm{CRing}_{\breve{\mathcal{O}}/}^{p\text{-comp}}$, an \textbf{$n$-truncated $(G, \mu)$-aperture} over $R$ is a $G$-torsor $\mathscr{Q}$ over $R^{\mathrm{syn}} \otimes \mathbb{Z}/p^n\mathbb{Z}$ satisfying the following boundedness condition: For every geometric point $R \to \kappa$ of $\mathrm{Spf}(R)$, the restriction of the pullback $(x_{\mathrm{dR}}^{\mathcal{N}})^*\mathscr{Q}$ over $B\mathbb{G}_m \times \mathrm{Spec}(\kappa)$ is isomorphic to that of $Q_\mu$. 

These organize into a derived $p$-adic formal prestack denoted $\mathrm{BT}^{G,\mu}_n$. We define the stack of $(G,\mu)$-apertures as the inverse limit
\[
\mathrm{BT}^{G,\mu}_\infty = \varprojlim_n \mathrm{BT}^{G,\mu}_n.
\]
\end{definition}

The fundamental representability result for these stacks is the following theorem of Madapusi and Youcis.

\begin{theorem}[Representability, \cite{MY2026} Theorem 3.1.5]
The formal prestack $\mathrm{BT}^{G,\mu}_n$ is represented by a zero-dimensional connected quasi-compact smooth $p$-adic formal Artin stack over $\breve{\mathcal{O}}$ with affine diagonal. Furthermore, the transition maps $\mathrm{BT}^{G,\mu}_{n+1} \to \mathrm{BT}^{G,\mu}_n$ are smooth and surjective.
\end{theorem}

In characteristic $p$, $1$-truncated apertures are closely related to the theory of $G$-zips. Let $k$ be the residue field of $\breve{\mathcal{O}}$. 

\begin{theorem}[Apertures to $G$-zips, \cite{MY2026} Proposition 3.5.5]
There is a natural smooth surjective map 
\[
\mathrm{BT}^{G,\mu}_1 \otimes_{\breve{\mathcal{O}}} k \to G\text{-}\mathrm{zip}_\mu
\]
of $0$-dimensional Artin stacks over $k$ that is a gerbe for a connected $p$-torsion finite flat commutative group scheme.
\end{theorem}

\subsection{The $\mu$-Ordinary Locus and Canonical Lifts}

\paragraph{Setup} Let $G$ be a reductive group scheme over $\mathbb{Z}_p$. Choose a maximal torus $T \subset G$ contained in a Borel subgroup of $G$. We may assume $\mu$ factors through $T_{\breve{\mathcal{O}}}$. Suppose $[\breve{\mathcal{O}}[1/p] : \mathbb{Q}_p] = d$, and let $\nu = \sum_{i=0}^{d-1} \varphi^i(\mu)$ be the associated cocharacter of $T$ defined over $\mathbb{Z}_p$. Let $P^+_\nu \subset G$ be the parabolic subgroup whose Lie algebra is the sum of the non-negative eigenspaces for $\nu$, and let $M_\nu$ be its Levi quotient.

\begin{definition}[The $\mu$-ordinary stratum]
Let $G\text{-}\mathrm{zip}_\mu^{\mathrm{ord}}$ denote the unique open dense stratum of $G\text{-}\mathrm{zip}_\mu$ associated with the maximal element in the refined Bruhat order of the Weyl group. For $n \in \mathbb{N} \cup \{\infty\}$, we define the \textbf{$\mu$-ordinary locus} $\mathrm{BT}^{G,\mu,\mathrm{ord}}_n \subset \mathrm{BT}^{G,\mu}_n$ as the unique open formal substack characterized by the property that its mod-$p$ fiber is
\[
\mathrm{BT}^{G,\mu,\mathrm{ord}}_n \otimes k = (\mathrm{BT}^{G,\mu}_n \otimes k) \times_{G\text{-}\mathrm{zip}_\mu} G\text{-}\mathrm{zip}_\mu^{\mathrm{ord}}.
\]
\end{definition}

A crucial feature of the $\mu$-ordinary locus is that it admits a canonical lift of the Frobenius endomorphism, descending from the parabolic subgroup $P^+_\nu$. Set $q_0 = p^d = |k|$.

\begin{proposition}[Canonical Frobenius lifting, \cite{MY2026} Proposition 3.6.11]
\label{prop:frob_lift}
There is a canonical endomorphism 
\[
\widetilde{\Phi} \colon \mathrm{BT}^{P^+_\nu, \mu}_n \to \mathrm{BT}^{P^+_\nu, \mu}_n
\]
of formal $\breve{\mathcal{O}}$-stacks lifting the $q_0$-Frobenius endomorphism of $\mathrm{BT}^{P^+_\nu, \mu}_n \otimes k(v)$.
\end{proposition}

The connection between the moduli of $P^+_\nu$-apertures and the ordinary locus of $G$-apertures is given by the following result, which shows that over the ordinary locus, the $G$-structure admits a canonical reduction to $P^+_\nu$.

\begin{proposition}[Reduction to parabolic, \cite{MY2026} Proposition 3.6.12]
\label{prop:ord_iso}
For $n \in \mathbb{N}$, the natural map $\mathrm{BT}^{P^+_\nu, \mu}_n \to \mathrm{BT}^{G,\mu}_n$ induces an isomorphism
\[
\mathrm{BT}^{P^+_\nu, \mu}_n \xrightarrow{\sim} \mathrm{BT}^{G,\mu,\mathrm{ord}}_n.
\]
\end{proposition}

Combined with Proposition \ref{prop:frob_lift}, Proposition \ref{prop:ord_iso} ensures that the $\mu$-ordinary locus of the stack of $G$-apertures comes equipped with a canonical pseudo-Frobenius lift. This formal endomorphism acts as the foundation for defining the canonical lift of ordinary points and building the anticanonical tower in the strict neighborhood of the ordinary locus.

\section{Global $\mu$-Ordinary Locus and $p$-Hecke Correspondences}

Having established the local theory of $G$-apertures, we now translate these structures to the global setting of integral canonical models. Throughout, we let $\mathscr{S}_K$ be a limpid integral canonical model (ICM) for the Shimura variety $\mathrm{Sh}_K$ over $\mathcal{O}_{E,(v)}$. 
\begin{definition}[Madapusi--Youcis \cite{MY2026}]
\label{def:integral_canonical_model}
Let $(G, X)$ be a Shimura datum with reflex field $E$, and fix a place $v \mid p$ of $E$. Let $K \subset G(\mathbb{A}_f)$ be a neat compact open subgroup of the form $K = K_p K^p$, where $K_p \subset G(\mathbb{Q}_p)$ is a hyperspecial subgroup. Let $\mathrm{Sh}_K$ denote the associated Shimura variety over $E$, and let $\mathrm{Et}_{K,p}$ be the canonical $G^c(\mathbb{Z}_p)$-local system over $\mathrm{Sh}_K$, where $G^c$ denotes the cuspidal quotient of $G$.

An integral model $\mathscr{S}_K$ for $\mathrm{Sh}_K$ over $\mathcal{O}_{E,(v)}$, with $v$-adic completion $\widehat{\mathscr{S}}_K$, is called an \textbf{integral canonical model} (ICM for short) if the following two conditions hold:

\begin{enumerate}
    \item \textbf{(Serre--Tate property)} There exists a formally \'etale map of $p$-adic formal stacks over $\mathrm{Spf}\, \mathcal{O}_{E,v}$
    \[
    \varpi \colon \widehat{\mathscr{S}}_K \to \mathrm{BT}^{G^c,-\mu_v}_\infty,
    \]
    such that the induced $G^c(\mathbb{Z}_p)$-local system on the generic fiber is isomorphic to $\mathrm{Et}_{K,p}$.
    
    \item \textbf{(Pointwise criterion)} For any mixed characteristic $(0,p)$ complete discrete valuation field $F$ over $\mathcal{O}_{E,v}$ with perfect residue field, and for any point $x \in \mathrm{Sh}_K(F)$, the following statements are equivalent:
    \begin{enumerate}
        \item[(a)] $x$ extends to an integral point $x \in \mathscr{S}_K(\mathcal{O}_F)$;
        \item[(b)] $\mathrm{Et}_{K,p}|_x$ is a crystalline $G^c(\mathbb{Z}_p)$-local system;
        \item[(c)] $\mathrm{Et}_{K,p}|_x$ is a potentially crystalline $G^c(\mathbb{Z}_p)$-local system.
    \end{enumerate}
\end{enumerate}
\end{definition}

\begin{definition}[Madapusi--Youcis \cite{MY2026}]
\label{def:limpid_icm}
Let $(G, X)$ be a Shimura datum with reflex field $E$. Let $p$ be a prime of good reduction, so that $G_{\mathbb{Q}_p}$ extends to a reductive group scheme over $\mathbb{Z}_p$, and let $v$ be a place of $E$ over $p$. Consider a compact open subgroup $K = K_p K^p \subset G(\mathbb{A}_f)$ of hyperspecial level at $p$, and let $\mathrm{Sh}_K$ denote the associated Shimura variety over $E$. 

Let $G^c$ denote the cuspidal quotient of $G$. For any prime $\ell \neq p$, let $K^c_\ell \subset G^c(\mathbb{Q}_\ell)$ be the compact open subgroup defined as the image of $K$ under the natural projection $G(\mathbb{A}_f) \to G^c(\mathbb{Q}_\ell)$. The prime-to-$p$ geometry of the Shimura variety naturally equips $\mathrm{Sh}_K$ with a $K^c_\ell$-local system, which we denote by $\mathrm{Et}_{K,\ell}$.

An integral canonical model (ICM) $\mathscr{S}_K$ for $\mathrm{Sh}_K$ over $\mathcal{O}_{E,(v)}$ is said to be \textbf{limpid} if, for every prime $\ell \neq p$, the $K^c_\ell$-local system $\mathrm{Et}_{K,\ell}$ on the generic fiber extends to a local system over the entire integral model $\mathscr{S}_K$.
\end{definition}
\begin{definition}[The global $\mu$-ordinary locus, \cite{MY2026} Definition 7.4.3]
The \textbf{$\mu$-ordinary locus} of the integral canonical model $\mathscr{S}_K$ is defined by the fiber product along the syntomic realization map $\varpi$:
\[
\widehat{\mathscr{S}}^{\mathrm{ord}}_K = \widehat{\mathscr{S}}_K \times_{\mathrm{BT}^{G,-\mu_v}_\infty} \mathrm{BT}^{G,-\mu_v,\mathrm{ord}}_\infty.
\]
\end{definition}

In order to construct the anticanonical tower, we must lift the Frobenius endomorphism globally. We begin by observing how the local canonical Frobenius lift interacts with the \'etale realization.

\begin{remark}[Analytic fiber of the canonical Frobenius lift, \cite{MY2026} Remark 7.4.8]
\label{rem:analytic_frob}
Recall the canonical $q_0$-Frobenius lift $\widetilde{\Phi}$ on $\mathrm{BT}^{P^+_\nu,-\mu_v}_n \simeq \mathrm{BT}^{G,-\mu_v,\mathrm{ord}}_n$. For any $p$-complete $\mathcal{O}_{E_v}$-algebra $R$, we have the following commuting diagram:
\[
\begin{array}{ccc}
\mathrm{BT}^{P^+_\nu,-\mu_v}_n(R) & \xrightarrow{T_{\mathrm{\acute{e}t}}} & \mathrm{Loc}_{P^+_\nu(\mathbb{Z}/p^n\mathbb{Z})}(R[1/p]) \\
\vspace{1mm} \\
\downarrow{\widetilde{\Phi}} & & \downarrow{\mathrm{int}(\nu(p)^{-1})} \\
\vspace{1mm} \\
\mathrm{BT}^{P^+_\nu,-\mu_v}_n(R) & \xrightarrow{T_{\mathrm{\acute{e}t}}} & \mathrm{Loc}_{P^+_\nu(\mathbb{Z}/p^n\mathbb{Z})}(R[1/p])
\end{array}
\]
Here, the right vertical arrow is induced by the endomorphism of $P^+_\nu(\mathbb{Z}/p^n\mathbb{Z})$ given by the adjoint action of $\nu(p)^{-1}$. 
\end{remark}

This $p$-adic behavior perfectly mirrors the behavior of global Hecke correspondences at $p$. 

\begin{remark}[Global $p$-Hecke correspondences, \cite{MY2026} Remark 7.4.9]
Fix an element $h \in G(\mathbb{Q}_p)$. Associated with this element and the subgroup $K_h = K \cap hKh^{-1}$, we have the $p$-Hecke correspondence:
\[
\mathrm{Sh}_{K_h} \xrightarrow{(s_h, t_h)} \mathrm{Sh}_K \times \mathrm{Sh}_K
\]
where the map $t_h$ is induced from the natural inclusion $K_h \subset K$, and $s_h$ is induced from conjugation by $h^{-1}$.
\end{remark}

By leveraging the correspondence between the adjoint action on the local systems (Remark \ref{rem:analytic_frob}) and the global $p$-Hecke correspondences, we obtain a canonical pseudo-Frobenius lift on the ordinary locus of the integral model.

\begin{lemma}[Canonical pseudo-Frobenius lift on the $\mu$-ordinary locus, \cite{MY2026} Lemma 7.4.10]
\label{lem:canonical_pseudo_frob}
There exists a canonical morphism $\widehat{\Phi}^{\mathrm{ord}}_K \colon \widehat{\mathscr{S}}^{\mathrm{ord}}_K \to \widehat{\mathscr{S}}^{\mathrm{ord}}_K$ such that the following diagram commutes:
\[
\begin{array}{ccc}
\widehat{\mathscr{S}}^{\mathrm{ord}}_K & \xrightarrow{\varpi} & \mathrm{BT}^{G,-\mu_v,\mathrm{ord}}_\infty \\
\vspace{1mm} \\
\downarrow{\widehat{\Phi}^{\mathrm{ord}}_K} & & \downarrow{\widetilde{\Phi}} \\
\vspace{1mm} \\
\widehat{\mathscr{S}}^{\mathrm{ord}}_K & \xrightarrow{\varpi} & \mathrm{BT}^{G,-\mu_v,\mathrm{ord}}_\infty
\end{array}
\]
Here, the right vertical arrow is the inverse limit over $n$ of the maps $\widetilde{\Phi}$. Furthermore, the reduction-mod-$p$ of $\widehat{\Phi}^{\mathrm{ord}}_K$ is a purely inseparable finite flat map.
\end{lemma}

\begin{remark}[Sections to the $p$-Hecke correspondence]
A key feature hidden within the proof of Lemma \ref{lem:canonical_pseudo_frob} is the existence of a geometric section. By taking $h = \nu(p)^{-1}$ and setting $K_\nu = K \cap \nu(p)K\nu(p)^{-1}$, the proof establishes that over the adic generic fiber of the ordinary locus $\widehat{\mathscr{S}}^{\mathrm{ord}}_{K,\eta}$, we have a canonical section to the $p$-Hecke correspondence:
\[
\widehat{\mathscr{S}}^{\mathrm{ord}}_{K,\eta} \to \mathrm{Sh}^{\mathrm{an}}_{K_\nu} \times_{t_\nu^{\mathrm{an}}, \mathrm{Sh}^{\mathrm{an}}_K} \widehat{\mathscr{S}}^{\mathrm{ord}}_{K,\eta}
\]
Composing this canonical section with the map $s_\nu^{\mathrm{an}}$ yields the adic fiber of the pseudo-Frobenius lift $\widehat{\Phi}^{\mathrm{ord}}_K$. This geometric section acts as the replacement for the canonical subgroup in constructing the anticanonical tower.
\end{remark}
\section{Generalized $\mu$-Ordinary Hasse Invariants via $G$-Zips}

To control the geometry of the ordinary locus in the special fiber of our limpid integral models, we rely on the group-theoretic construction of Hasse invariants developed by Koskivirta and Wedhorn \cite{KW2018}. We briefly review their construction for stacks of $G$-zips and then apply it to our setting.

\subsection{The Stack of $G$-Zips and Hasse Invariants}

Let $k$ be an algebraic closure of $\mathbb{F}_p$ (or $\mathbb{F}_q$). We fix a \textbf{zip datum} $\mathcal{Z} = (G, P, L, Q, M, \phi)$, where $G$ is a connected reductive group over $\mathbb{F}_q$, $\phi \colon G \to G$ is the relative $q$-power Frobenius, $P$ and $Q$ are parabolic subgroups of $G_k$, and $L \subset P$ and $M \subset Q$ are Levi subgroups such that $\sigma(L) = M$ (where $\sigma$ is the pullback by the absolute Frobenius). 

\begin{definition}[$G$-Zips, \cite{KW2018} \S 1.1]
The \textbf{zip group} $E$ associated with the zip datum $\mathcal{Z}$ is defined as
\[
E := \{ (x, y) \in P \times Q \mid \phi(\overline{x}) = \overline{y} \},
\]
where $\overline{x} \in L$ and $\overline{y} \in M$ are the Levi components of $x$ and $y$. The group $E$ acts on $G$ via $(a,b) \cdot g := a g b^{-1}$. The \textbf{stack of $G$-zips} of type $\mathcal{Z}$ is the quotient stack 
\[
G\text{-}\mathrm{Zip}^{\mathcal{Z}} \simeq [E \backslash G].
\]
\end{definition}

Given a cocharacter $\mu \colon \mathbb{G}_{m,k} \to G_k$, one can naturally attach a zip datum $\mathcal{Z}_\mu$. We write $G\text{-}\mathrm{Zip}^\mu$ for the associated stack. Let $U_\mu \subset G\text{-}\mathrm{Zip}^\mu$ denote the unique open zip stratum, which corresponds to the $\mu$-ordinary locus. To any character $\lambda \in X^*(L)$, one associates a line bundle $\mathscr{V}(\lambda)$ on $G\text{-}\mathrm{Zip}^\mu$.

\begin{definition}[$\mathcal{Z}$-ample characters, \cite{KW2018} Definition 5.1.1]
A character $\lambda \in X^*(L)$ is called \textbf{$\mathcal{Z}$-ample} if the attached line bundle on $G/\sigma^{-1}Q$ is anti-ample (equivalently, $\langle \lambda, \alpha^\vee \rangle < 0$ for all $\alpha \in \Delta \setminus \sigma^{-1}J$, where $J$ is the type of $Q$).
\end{definition}

\begin{definition}[Hasse Invariant, \cite{KW2018} Definition 5.1.3]
Let $\lambda \in X^*(L)$. A global section $f \in H^0(G\text{-}\mathrm{Zip}^\mu, \mathscr{V}(\lambda))$ is called a \textbf{Hasse invariant} if its non-vanishing locus is exactly the open stratum $U_\mu$.
\end{definition}

The main result of Koskivirta and Wedhorn establishes the existence of such invariants purely group-theoretically for a suitable multiple of any $\mathcal{Z}$-ample character.

\begin{theorem}[Koskivirta--Wedhorn, \cite{KW2018} Theorem 5.1.4]
\label{thm:KW_hasse}
If $\lambda \in X^*(L)$ is $\mathcal{Z}$-ample, there exists a Hasse invariant 
\[
h \in H^0(G\text{-}\mathrm{Zip}^\mu, \mathscr{V}(N_\mu \lambda)),
\]
where $N_\mu \ge 1$ is an explicit integer depending only on the finite group associated with the stabilizer of the open stratum. Moreover, $h$ is unique up to a non-zero scalar.
\end{theorem}

\subsection{Hasse Invariants for Limpid Integral Models}

We now transfer this group-theoretic result to our geometric setting. Recall that for an unramified Shimura tuple $(G, \mathscr{G}, X, K)$ admitting a limpid integral canonical model $\mathscr{S}_K$ over $\mathcal{O}_{E,(v)}$, the syntomic realization yields a formally smooth morphism 
\[
\varpi \colon \widehat{\mathscr{S}}_K \to \mathrm{BT}^{G^c, -\mu_v}_\infty.
\]
By reducing modulo $p$ (specifically, over the residue field $k(v)$) and applying the truncation map to level 1 apertures alongside Proposition 3.5.5 of \cite{MY2026}, we obtain a smooth morphism of stacks:
\[
\zeta \colon \mathscr{S}_K \otimes k(v) \longrightarrow \mathrm{BT}^{G^c, -\mu_v}_1 \otimes k(v) \longrightarrow G^c\text{-}\mathrm{Zip}_{-\mu_v}.
\]
By construction, the inverse image of the open zip stratum $U_{-\mu_v}$ under $\zeta$ is precisely the $\mu$-ordinary locus $\mathscr{S}_K^{\mathrm{ord}} \otimes k(v)$ of the integral model.

For a Shimura variety of Hodge type, one typically pulls back the Hodge line bundle $\omega$ from the Siegel moduli space. In our abstract setting—where no universal abelian variety exists—we substitute the Hodge bundle with the pullback of a $\mathcal{Z}$-ample line bundle from the stack of $G$-zips. 

Let $\lambda \in X^*(L)$ be a $\mathcal{Z}$-ample character for the zip datum attached to $-\mu_v$, and let $\mathscr{V}(\lambda)$ be the corresponding line bundle on $G^c\text{-}\mathrm{Zip}_{-\mu_v}$. We define a line bundle on the special fiber of our limpid integral model by
\[
\mathscr{L} := \zeta^*(\mathscr{V}(\lambda)).
\]
Using Theorem \ref{thm:KW_hasse} and the smoothness of $\zeta$, we immediately deduce the following analogue of Corollary 5.6.1 of \cite{KW2018}:

\begin{corollary}[Global Hasse Invariant for Limpid ICMs]
\label{cor:hasse_limpid}
There exists a global section 
\[
h_K = \zeta^*(h) \in H^0(\mathscr{S}_K \otimes k(v), \mathscr{L}^{\otimes N_{\mu_v}})
\]
whose non-vanishing locus is exactly the $\mu$-ordinary locus $\mathscr{S}_K^{\mathrm{ord}} \otimes k(v)$.
\end{corollary}

\section{Extension of the $p$-Hecke Correspondence and the Anticanonical Tower}

In the previous section, we established the existence of the global Hasse invariant $h_K$ and the canonical pseudo-Frobenius lift $\widehat{\Phi}_K^{\mathrm{ord}}$ over the formal $\mu$-ordinary locus. We now analytically continue this structure to a strict neighborhood of the ordinary locus in the adic generic fiber, which will allow us to construct the anticanonical tower.

We fix our conventions for the $p$-Hecke correspondences. Let $K_0 = G(\mathbb{Z}_p)$ be the hyperspecial maximal compact subgroup at $p$, and recall the element $\nu(p)^{-1} \in G(\mathbb{Q}_p)$ generating the canonical $p$-Hecke correspondence. We define a decreasing sequence of compact open subgroups
\[
    K_m = K_0 \cap \nu(p)^m K_0 \nu(p)^{-m}.
\]
The corresponding Shimura varieties at finite level are denoted $\mathcal{H}_m = \mathrm{Sh}_{K^p K_m}$. For each $m \geq 1$, we have the target projection $u_m \colon \mathcal{H}_m \to \mathrm{Sh}_K$ (the standard level-forgetting map) and the source projection $s_m \colon \mathcal{H}_m \to \mathrm{Sh}_K$.

Fix a finite unramified extension $E/\mathbb Q_p$ with residue field
$k=\mathbb F_q$, where $q=p^d$, and a complete algebraically closed
extension $C/E$. Normalize $v(p)=1$. For rational exponents used below,
choose compatible elements $p^r\in C$; in particular
$(p^{r/q})^q=p^r$.

Fix a Hasse invariant $\bar h\in H^0(\cS_k,\mathscr L)$ whose
non-vanishing locus is the $\mu$-ordinary stratum. Choose formal affine charts on which $\mathscr L$ is trivial on the
special fiber, and lift its local equations to functions $h\in A$.
For $0\leq r<1$ define $S(r)$ locally by $|h|\geq |p|^r$.
On overlaps two equations differ by an integral unit modulo $p$.
Their norms therefore give the same inequality for $r<1$, so these
domains glue. The ordinary tube is $S(0)$. 

\begin{lemma}[Extension of a finite etale section]\label{lem:section}
Let $V$ be a reduced affinoid analytic space, let $h\in\mathcal O^+(V)$,
and set $V(r)=\{|h|\geq |p|^r\}$. A section of a finite etale
morphism $W\to V$ over $V(0)$ extends over $V(r)$ for some $r>0$.
The extension is unique as a germ along $V(0)$.
\end{lemma}
\begin{proof}
Write $V=\Spa(R,R^+)$ and let $D$ be the finite etale $R$-algebra
of $W$. The section determines an idempotent in
$D\widehat\otimes_R\mathcal O(V(0))$ cutting out a rank-one factor.
The image of $D[1/h]$ is dense in this algebra. Approximate the
idempotent by an element $b\in D[1/h]$, which is analytic on a strict
neighborhood.

The Berkovich spectra of the nested affinoids $V(r)$ are compact and
their intersection is the spectrum of $V(0)$. For a fixed analytic
function on one of these neighborhoods its supremum norms consequently
converge to its norm on $V(0)$. The same applies on the finite cover.
After shrinking, $b^2-b$ is small enough; Hensel's lemma guarantees that the standard Newton iteration
\[
    b_{n+1} = b_n - f(b_n)f'(b_n)^{-1}
\]
converges with respect to the Banach norm to a unique exact root $e_\varepsilon \in \mathcal{B}_\varepsilon$ of $T^2 - T$ such that $\|e_\varepsilon - b\|_{\mathcal{B}_\varepsilon} < 1$. 

The finite projective algebra defined by $e_\varepsilon$ has rank 1 along $S(0)$, and since the rank of a finite locally free module over a reduced affinoid space is locally constant, we can shrink $\varepsilon$ further until it has rank 1 everywhere, defining the extended section over $S(\varepsilon)$. 
\end{proof}

On a finite affinoid cover, uniqueness of germs allows the local
extensions to glue after choosing a common radius. Thus we get an analytic Hecke section
on some $S(r_0)$. Its other projection is an analytic map $\Phi$.
Lemma~\ref{lem:global-height} below proves forward preservation of a
finite collection of formal affine charts after shrinking.

Put $A_C=A\widehat\otimes_{\OEv}\OC$.
and
\[
 B_r=A_C\langle u_r\rangle/(hu_r-p^r)
\]

\begin{lemma}[Flatness of the models]\label{lem:flat}
Each $B_r$ is $\OC$-flat and $p$-adically complete. For $0<s\leq r$
restriction is induced by
\[
 B_r\longrightarrow B_s,\qquad u_r\longmapsto p^{r-s}u_s.
\]
After inverting $p$, $B_r$ is the affinoid algebra of $U(r)$.
\end{lemma}
\begin{proof}
The usual rational localization presentation gives the last assertion
and the displayed restriction map. For flatness, $h$ is a
non-zero-divisor after reduction modulo any nonzero scalar in $\OC$.
This follows first over finite extensions of $E$, by induction on the
valuation using the hypothesis on $\bar h$, and then by flat base
change and completion. In the polynomial presentation, multiplication
by $hu-p^r$ remains injective modulo every such scalar: compare the
leading coefficient in $u$. The quotient therefore has no scalar
torsion. Taking its $p$-adic completion preserves this property; one
may use cofinal scalar powers to compute the completion. A torsion-free
module over the valuation ring $\OC$ is flat.
\end{proof}

\begin{lemma}[Integral clearing of poles]\label{lem:poles}
Let $F$ be an analytic function on $U(r_0)$, defined over a finite
extension of $E$, whose restriction to $U(0)$ belongs to
$p\widehat{A[1/h]}$ after the same scalar extension. There are an
integer $N\geq0$ and a smaller positive radius such that
\[
 F\in p^{1-Nr}B_r
\]
for all smaller rational $r>0$ with $Nr<1$.
\end{lemma}
\begin{proof}
Choose a finite extension $E'/E$ containing $p^{r_0}$ and the field
of definition of $F$, and put
$A'=A\widehat\otimes_{\OEv}\mathcal O_{E'}$.
Use the analogous model over $\mathcal O_{E'}$ temporarily.
Finite Laurent expressions in $h$ are dense in $B_{r_0}[1/p]$ for
the topology defined by this model. Choose
$f\in A'[1/h,1/p]$ such that
\[
 F-f\in p^2B_{r_0}.
\]
Restriction to the ordinary model takes $B_{r_0}$ into
$\widehat{A'[1/h]}$. Hence $f\in p\widehat{A'[1/h]}$.
The natural map
\[
 A'[1/h]/p^n\longrightarrow\widehat{A'[1/h]}/p^n
\]
is an isomorphism for every $n$. Clearing a power of $p$ in $f$
therefore shows that $f\in pA'[1/h]$. Write $f=pa/h^N$ with
$a\in A'$.

For $r\leq r_0$, restriction takes $B_{r_0}$ into $B_r$ and gives
\[
 F=p^{1-Nr}a u_r^N+p^2z_r,\qquad z_r\in B_r.
\]
This proves the assertion after scalar extension to $\OC$.
Taking the largest denominator exponent proves the finite-family
version.
\end{proof}

\begin{proposition}[Frobenius congruence on original coordinates]
\label{prop:coordinates}
There are $N\geq0$ and $r_1>0$ such that
\[
 \Phi^*f-f^q\in p^{1-Nr}B_r\qquad(f\in A,
 \ 0<r\leq r_1,
 \ Nr<1).
\]
The expression $f$ on the right is restriction from the original
formal chart.
\end{proposition}
\begin{proof}
Choose finitely many topological $\OEv$-algebra generators of $A$
and also include $h$ in the list. For each of these elements,
$\Phi^*f-f^q$ restricts to $p\widehat{A[1/h]}$. Apply Lemma~\ref{lem:poles} simultaneously.
Modulo $p^{1-Nr}$ the $q$-power operation is additive and
multiplicative. Moreover, $c^q\equiv c\pmod p$ for $c\in\OEv$.
The congruence therefore extends to polynomials in the chosen
generators, and then by $p$-adic continuity to $A$.
\end{proof}

\begin{corollary}[Height formula]\label{cor:height}
If $(N+q)r<1$, then
\[
 |h(\Phi(x))|=|h(x)|^q\qquad(x\in U(r)).
\]
Consequently, for sufficiently small $\eps>0$, $\Phi$ maps
$U(\eps/q)$ into $U(\eps)$.
\end{corollary}
\begin{proof}
Write $\Phi^*h=h^q+e$, with $e\in p^{1-Nr}B_r$.
At any point of $U(r)$, $|h|^q\geq |p|^{qr}>|p|^{1-Nr}\geq|e|$.
The ultrametric inequality proves the assertion. This argument proves
containment, not surjectivity or finiteness.
\end{proof}

Let $\cS_K/\OEv$ be the canonical integral model of the Shimura variety. For the compactness arguments We use the associated Berkovich spaces.

\begin{lemma}[Forward charts and the global height]
\label{lem:global-height}
There is a continuous function $w:S^{\mathrm{Berk}}\to[0,1]$ such that
$S(r)=\{w\leq r\}$ for $0\leq r<1$ and $S(0)=\{w=0\}$.
After shrinking the radius of definition of $c$, one has
\begin{equation}\label{eq:global-height}
 w(\Phi(x))=q\,w(x)\qquad(x\in S(b))
\end{equation}
for some $b>0$ with $qb<1$.
Moreover, a finite affine cover of $\cS$ can be chosen such that, if
$U_i$ denotes the generic fiber of the completion of its $i$-th chart,
then $\Phi(U_i\cap S(b))\subset U_i$.
\end{lemma}
\begin{proof}
Choose a finite affine cover $\mathscr U_i=\operatorname{Spec}R_i$ of
$\cS$ on whose special fibers the Hasse line is trivial. Choose local
equations $h_i\in R_i$ lifting the Hasse invariant. On the compact
analytic domain $U_i$ put
\[
 w_i(x)=\min\{v(h_i(x)),1\},\qquad v(0)=+\infty.
\]
These are continuous. On overlaps the equations differ by a unit
modulo $p$; the unit and its inverse have norm at most one. Thus the
truncated valuations agree. Since $\cS$ is proper, the $U_i$ cover
$S^{\mathrm{Berk}}$ and form a finite closed cover of this compact
Hausdorff space. The $w_i$ glue to a continuous $w$ with the stated
sublevel sets.

Put $U_i(r)=U_i\cap S(r)$. The formal ordinary Frobenius lift has
the same underlying special-fiber map as $q$-Frobenius, which is the
identity on the underlying topological space. Consequently,
$\Phi(U_i(0))\subset U_i$. In particular this image is contained in
the Berkovich open subset $(\mathscr U_{i,E})^{\mathrm{Berk}}$ of
$S^{\mathrm{Berk}}$.
The compact sets $U_i(r)$ decrease to $U_i(0)$. Continuity of $\Phi$
therefore implies, after shrinking, that
\[
 \Phi(U_i(r))\subset(\mathscr U_{i,E})^{\mathrm{Berk}}.
\]
This statement only concerns the algebraic open subset and does not
yet claim that its integral-coordinate domain $U_i$ is preserved.

Choose finitely many $\OEv$-algebra generators $f_{ij}$ of $R_i$,
and add $h_i$ to the list. Their pullbacks by $\Phi$ are now analytic
on $U_i(r)$. On $U_i(0)$ the functions
$\Phi^*f_{ij}-f_{ij}^q$ have norm at most $|p|$, by the formal
Frobenius property. Fix $0<\gamma<1$. The shrinking-norm argument
on the nested compact sets gives, on a common smaller radius,
\[
 \|\Phi^*f_{ij}-f_{ij}^q\|_{U_i(r)}\leq|p|^\gamma
\]
for all $i,j$. Since the source coordinates have norm at most one,
all $\Phi^*f_{ij}$ have norm at most one. Inside
$(\mathscr U_{i,E})^{\mathrm{Berk}}$ these inequalities describe
the generic fiber of the formal completion, so $\Phi(U_i(r))\subset U_i$.

The same estimate for $h_i$ gives
$\Phi^*h_i=h_i^q+e_i$ with $\|e_i\|\leq|p|^\gamma$.
Choose $b>0$ smaller than all the preceding radii and with $qb<\gamma$.
Then $|h_i(x)|^q>|e_i(x)|$ on $U_i(b)$, proving
$v(h_i(\Phi(x)))=qv(h_i(x))$. Both sides are less than one, so this
is precisely \eqref{eq:global-height}. The finitely many charts give
one radius $b$ valid on all of $S(b)$.
\end{proof}

\begin{proposition}[Finite etale surjective ordinary branch]
\label{prop:finite-branch}
Let $n=\dim S$. There is a rational $\eps_*>0$ such that for every
rational $0<r\leq\eps_*$ the map
\[
 \Phi:S(r/q)\longrightarrow S(r)
\]
is finite etale and surjective of degree $q^n$. Thus one choice of
$0<\eps\leq\eps_*$ works simultaneously for all the transitions
\[
 \Phi:S(\eps/q^{m+1})\longrightarrow S(\eps/q^m),\qquad m\geq0.
\]
More precisely, $c(S(r/q))$ is an open-and-closed analytic subspace of
$t^{-1}(S(r))$, and all these subspaces are obtained by restricting
one fixed finite etale factor.
\end{proposition}
\begin{proof}
Choose $b$ as in Lemma~\ref{lem:global-height} and fix $0<a<b$.
The set $V=\{w<a\}$ is a Berkovich open neighborhood of the ordinary
tube, contained in the domain of $c$. Since $s$ is finite etale, its
section over $V$ is an open-and-closed immersion into $s^{-1}(V)$.
In particular,
\[
 \Omega=c(V)
\]
is an open analytic subspace of $\cH^{\mathrm{Berk}}$.

Fix a positive rational $R<\min\{a,1\}$. In particular $R/q<a$.
The height identity gives an equality of analytic domains
\begin{equation}\label{eq:finite-factor}
 Z_R:=\Omega\cap t^{-1}(S(R))=c(S(R/q)).
\end{equation}
Indeed a point of $\Omega$ is uniquely $c(x)$ for $x\in V$, and
$t(c(x))\in S(R)$ is equivalent to $qw(x)\leq R$.
The left side of \eqref{eq:finite-factor} is open in $t^{-1}(S(R))$.
The right side is compact, since $S(R/q)$ is compact. The ambient
Berkovich space is Hausdorff, so this image is closed as well.
It follows that $Z_R$ is an open-and-closed factor of the finite
etale space $t^{-1}(S(R))$ over $S(R)$. Hence
\[
 Z_R\longrightarrow S(R)
\]
is finite etale. Via $c$ it is exactly $\Phi:S(R/q)\to S(R)$.

Let $B\subset S(R)^{\mathrm{Berk}}$ be the locus where this rank
differs from $q^n$, including rank zero. The rank function is locally
constant, so $B$ is closed in the compact space $S(R)^{\mathrm{Berk}}$.
It is disjoint from $S(0)=\{w=0\}$. If $B$ is nonempty, continuity
of $w$ gives $\min_B w>0$. Choose rational $0<\eps_*\leq R$ with
$\eps_*<\min_B w$; if $B$ is empty any such $\eps_*$ works.
The restriction of $Z_R$ to $S(\eps_*)$ now has constant rank $q^n$.

For every $0<r\leq\eps_*$, the height identity again gives
\[
 Z_R\times_{S(R)}S(r)=c(S(r/q)).
\]
\end{proof}

Choose a rational $\eps>0$ so small that $\eps\leq r_1$ and
$(N+q)\eps<1/2$. Set
\[
 r_m=\eps/q^m,\quad a_m=p^{r_m},\quad
 B_m=B_{r_m},\quad u_m=u_{r_m},\quad \delta=1/2.
\]
Thus $a_{m+1}^q=a_m$.

\begin{lemma}[The transition on the rational generator]
\label{lem:transition}
The analytic pullback induces a continuous $\OC$-algebra map
$\tau_m:B_m\to B_{m+1}$, and modulo $p^\delta$ it satisfies
\[
 \tau_m(f)=f^q\quad(f\in A),\qquad
 \tau_m(u_m)=u_{m+1}^q,\qquad
 \tau_m(c)=c\quad(c\in\OC).
\]
\end{lemma}
\begin{proof}
Write $\Phi^*h=h^q+e$ on $U(r_{m+1})$, where
$e\in p^{1-Nr_{m+1}}B_{m+1}$. In $B_{m+1}[1/p]$ put
\[
 z=\frac{e}{h^q}=\frac{e u_{m+1}^q}{a_m}
 \in p^{1-(N+q)r_{m+1}}B_{m+1}\subset p^\delta B_{m+1}.
\]
Then $1+z$ is a unit in $B_{m+1}$ and
\begin{equation}\label{eq:u-transition}
 \Phi^*u_m=\frac{a_m}{\Phi^*h}
 =u_{m+1}^q(1+z)^{-1}.
\end{equation}
Proposition~\ref{prop:coordinates} shows that the images of the
original coordinates also belong to $B_{m+1}$. Together these formulas
define $\tau_m$ on the whole model and prove the congruences.
The map fixes $\OC$; it is not absolute Frobenius on the scalars.
\end{proof}

\begin{lemma}[Coefficient twist]\label{lem:twist}
There are ring isomorphisms
\[
 \kappa_m:B_m/p^\delta\xrightarrow{\ \sim\ }
 B_{m+1}/p^{\delta/q}
\]
which act as inverse $q$-power Frobenius on coefficients in $\OC$,
as the identity on the copy of $A/p$, and send $u_m$ to $u_{m+1}$.
They satisfy
\[
 F_q\kappa_m=\overline\tau_m
 \quad\text{and}\quad
 \kappa_{m+1}\overline\tau_m
 =\overline\tau_{m+1}\kappa_m,
\]
where the second equality uses the appropriate reductions on each
side.
\end{lemma}
\begin{proof}
Because $\delta\leq1$, reduction has the presentation
\[
 B_m/p^\delta=
 \bigl((A/p)\otimes_k(\OC/p^\delta)\bigr)[u_m]
 /(\bar h u_m-\bar a_m).
\]
Indeed, restricted power series become polynomials after reduction
modulo $p^\delta$. The map
\[
 F_q:\OC/p^{\delta/q}\xrightarrow{\ \sim\ }\OC/p^\delta
\]
is an isomorphism: surjectivity follows from algebraic closedness of
$C$, and injectivity follows from valuations. Its inverse fixes
$k=\mathbb F_q$ and takes $a_m$ to $a_{m+1}$. It therefore induces
the asserted isomorphism of the displayed presentations.
The two identities can be checked on $A/p$, $u_m$, and the coefficient
ring, using Lemma~\ref{lem:transition}.
\end{proof}

\begin{proposition}[Frobenius on the completed limit]
\label{prop:limit-frob}
Put
\[
 D=\colim_{m,\tau_m}B_m,\qquad B_\infty=\widehat D^{\,p}.
\]
Then $B_\infty$ is $\OC$-flat and Frobenius induces an isomorphism
\begin{equation}\label{eq:q-iso}
 F_q:B_\infty/p^{\delta/q}\xrightarrow{\ \sim\ }
 B_\infty/p^\delta.
\end{equation}
\end{proposition}
\begin{proof}
Filtered colimits of flat $\OC$-modules are flat. The $p$-adic
completion of a torsion-free $\OC$-module is torsion-free: for a
scalar $b\neq0$, compute using a cofinal sequence of scalar powers
divisible by $b$ and use cancellation before completion.
Hence both $D$ and $B_\infty$ are $\OC$-flat. The same argument
gives
\[
 B_\infty/p^s=D/p^s=\colim_m B_m/p^s
 \qquad(s>0).
\]

The commuting isomorphisms $\kappa_m$ induce an isomorphism
\[
 \kappa:\colim_m B_m/p^\delta
 \xrightarrow{\ \sim\ }
 \colim_m B_{m+1}/p^{\delta/q}.
\]
Dropping the first term does not change the filtered colimit. Since
$F_q\kappa_m=\overline\tau_m$, the composite $F_q\kappa$ is the
identity on $D/p^\delta$. Thus $F_q=\kappa^{-1}$, giving an
isomorphism, not merely a surjection. Passing to the completion gives
\eqref{eq:q-iso}.
\end{proof}

\begin{theorem}[Perfectoid algebra and the inverse limit]
\label{thm:local}
The algebra $R_\infty=B_\infty[1/p]$, with the topology defined by
$B_\infty$, is a perfectoid $C$-algebra. If $R_\infty^+$ is the
integral closure of $B_\infty$ in $R_\infty$, then
\[
 U_\infty=\Spa(R_\infty,R_\infty^+)
\]
is an affinoid perfectoid space representing
$\ilim_m U(r_m)^\diamond$.
\end{theorem}
\begin{proof}
We first reduce \eqref{eq:q-iso} to the usual $p$-power criterion.
Put $\beta=\delta/p^{d-1}$. The map
\begin{equation}\label{eq:p-iso}
 F_p:B_\infty/p^{\beta/p}\longrightarrow B_\infty/p^\beta
\end{equation}
is surjective because surjectivity of $F_q$ modulo $p^\delta$
implies it after reduction modulo $p^\beta$: a $q$-th root yields a
$p$-th root by raising to $q/p$.
For injectivity, if $x^p\in p^\beta B_\infty$, then
\[
 x^q\in p^{\beta p^{d-1}}B_\infty=p^\delta B_\infty.
\]
Injectivity in \eqref{eq:q-iso} gives
$x\in p^{\delta/q}B_\infty=p^{\beta/p}B_\infty$.
Thus \eqref{eq:p-iso} is an isomorphism.

Apply the flat integral almost-algebra criterion
\cite[Definition 5.1(ii), Lemma 5.6]{Scholze2012}, with
$\varpi=p^\beta$. The almost algebra $B_\infty^a$ is flat and
$\varpi$-adically complete, and \eqref{eq:p-iso} supplies the
required Frobenius isomorphism. It follows that
$(B_\infty^a)_*[1/p]$ is perfectoid and has power-bounded subring
$(B_\infty^a)_*$. This generic algebra and its topology agree with
$B_\infty[1/p]$: almost elements lie in $p^{-1}B_\infty$, and
$B_\infty\subset(B_\infty^a)_*$. This proves both completeness and
uniformity with the required Frobenius condition. In particular,
$R_\infty^+$ is an open bounded integrally closed subring of the
power-bounded elements, so it defines an affinoid perfectoid pair.

For completeness, the inverse-limit assertion can be checked directly
on affinoid perfectoid test spaces $T=\Spa(Q,Q^+)$. A compatible
family $T\to U(r_m)$ gives compatible maps $B_m\to Q^+$. These
induce a map $D\to Q^+$ continuous for the $p$-adic topologies,
which extends to $B_\infty$ since $Q^+$ is complete. After inverting
$p$ it extends across $R_\infty^+$ because $Q^+$ is integrally
closed in $Q$. Conversely such a map restricts to every finite stage.
These constructions are inverse and commute with restriction on test
spaces.
\end{proof}

\section{Anti-canonical tower}
\label{sec:tower}
In this section we want to identify the open neighborhoods of the ordinary locus with the anticanonical tower.
\begin{theorem}[Geometric anticanonical tower]\label{thm:tower}
For one sufficiently small rational $\eps>0$, the following hold:

There are analytic maps
\[
 j_m:S(\eps/q^m)\longrightarrow\cH_m
\]
with
\[
 s_mj_m=\mathrm{id},\qquad t_mj_m=\Phi^m,\qquad
 \rho_mj_{m+1}=j_m\Phi,
\]
where $\rho_m:\cH_{m+1}\to\cH_m$ forgets level.
\end{theorem}

\begin{lemma}[Straight Hecke chains, with their level data]
\label{lem:straight-chains}
Let $\mathcal C_m$ be the convolution of $m$ copies of $\cH_1$,
where consecutive arrows satisfy $t_1(h_i)=s_1(h_{i+1})$.
There is an open-and-closed immersion
\[
 \iota_m:\cH_m\longrightarrow\mathcal C_m.
\]
For $m\geq2$, its image consists precisely of those chains for which
every consecutive two-arrow subchain belongs to
$\iota_2(\cH_2)\subset\mathcal C_2$.
Under this identification $\rho_m$ drops the first arrow.
\end{lemma}
\begin{proof}
Let $\mathcal B$ be the enlarged Bruhat--Tits building of
$G_{\mathbb Q_p}$ and choose the hyperspecial point $o$ corresponding
to $K$, in an apartment containing $\nu$. Fix a Weyl-invariant
Euclidean metric, including the central vector factor. The building
is CAT(0). We use Cartan decomposition and the apartment properties
of the building; see~\cite{Yu2009}. The local-to-global geodesic
property is~\cite[II.1.4]{BH1999}.

The stabilizer of $o$ is exactly $K$. Indeed, in a Cartan
decomposition $g=k_1\lambda(p)k_2$, the cocharacter $\lambda$
translates $o$ by its vector in the enlarged apartment. If $g$
fixes $o$, that vector is zero, so $\lambda=0$ and $g\in K$.
Consequently $G(\mathbb Q_p)/K\to G(\mathbb Q_p)o$ is a bijection.
Using the enlarged building here retains central translations.
Put $\ell=d(o,ao)$; the points $a^io$ lie at equal intervals on a
line. An isometry fixing the endpoints of a segment fixes its unique
geodesic pointwise. Hence
\begin{equation}\label{eq:segment-stabilizer}
 K\cap a^{-m}Ka^m=\bigcap_{i=0}^m a^{-i}Ka^i.
\end{equation}
In particular $K_{m+1}\subset K_m$.

The map $\iota_m$ is, on the $p$-component of a level representative,
\begin{equation}\label{eq:chain-map}
 gK_m\longmapsto
 (ga^{-(m-1)}K_1,ga^{-(m-2)}K_1,\ldots,gK_1).
\end{equation}
The other adelic and symmetric-domain data are retained. Equation
\eqref{eq:segment-stabilizer} makes every component well-defined.
The vertices of this chain are
\[
 ga^{-m}K,ga^{-(m-1)}K,\ldots,gK.
\]
The source of the chain is $s_m$, and its target is $t_m$.
Both $\cH_m$ and $\mathcal C_m$ are finite etale over their initial
source. We may therefore check all the assertions on geometric
fibers. Trivializing the initial $K$-torsor identifies these fibers
with chains of $K$-cosets whose adjacent relative position is $KaK$.
This is the finite-cover description of Hecke correspondences
in~\cite[Remark 7.4.9]{MY2026}. It keeps the Hecke level structures;
we are not identifying a correspondence with its image in $S\times S$.

The map~\eqref{eq:chain-map} is injective on these fibers: equality
of the chain means that the change of representative belongs to
all the stabilizers in~\eqref{eq:segment-stabilizer}. It is therefore
an open-and-closed immersion of finite etale covers.
Write $\mathcal C_2^{\mathrm{str}}=\iota_2(\cH_2)$.
A two-arrow chain in this subcover is a geodesic segment of length
$2\ell$ in $\mathcal B$. If every consecutive two-arrow subchain
is in this subcover, the piecewise geodesic chain is a local geodesic,
and hence a geodesic. All its vertices lie in an apartment containing
its endpoints. Its vector displacement is $m$ times that of its
first segment, not just a vector of the same scalar length. Cartan
decomposition consequently gives endpoint relative position $Ka^mK$.
For endpoints of this relative position, the intermediate vertices
are the unique points at distances $i\ell$ on their geodesic.
They are exactly the cosets in~\eqref{eq:chain-map}, since
$G(\mathbb Q_p)/K$ embeds into $\mathcal B$.
This proves the claimed equality on geometric fibers and therefore
on finite etale covers. Conversely every chain in the image clearly
has the stated two-arrow property. If $\ell=0$, then $a\in K$ and
all the statements reduce to the identity correspondence.
Finally, dropping the first component in~\eqref{eq:chain-map}
is precisely the map induced by $K_{m+1}\subset K_m$.
\end{proof}

\begin{lemma}[Uniform continuation of the ordinary chains]
\label{lem:uniform-chains}
For one rational $\eps>0$, independent of $m$, the iterated section
on $S(\eps/q^m)$ factors uniquely through $\iota_m(\cH_m)$.
The resulting maps $j_m$ satisfy Theorem~\ref{thm:tower}.
\end{lemma}
\begin{proof}
First consider the ordinary tube. Let
$P=\mathcal P_\nu^+(\mathbb Z_p)$ be its canonical parabolic reduction.
The ordinary section is constructed from the $P$-fixed coset $aK$:
\[
 PaP=aP,\qquad a^{-1}Pa\subset P.
\]
Moreover, the parabolic reduction at $\Phi(x)$ is obtained from that
at $x$ by conjugation by $a^{-1}$, as in
\cite[Remark 7.4.8 and the proof of Lemma 7.4.10]{MY2026}.
In particular, transport of a chosen parabolic frame through the
first canonical modification is a parabolic frame for the second.
In a common rational frame the two modifications therefore have
vertices $K,aK,a^2K$. Changing the initial frame by $p\in P$
changes the transported frame by $a^{-1}pa\in P$, so this description
is independent of the frame. Thus
\[
 x\longmapsto(c(x),c(\Phi(x)))
\]
lands in $\mathcal C_2^{\mathrm{str}}$ on $S(0)$.

Choose a fixed rational $R>0$ with $qR\leq b$, so this two-step map
is defined on the compact space $S(R)^{\mathrm{Berk}}$.
The inverse image $D$ of the complement of
$\mathcal C_2^{\mathrm{str}}$ is open and closed in $S(R)$.
It is compact and disjoint from $S(0)$. If $D$ is nonempty,
$\min_D w>0$. Choose rational
\[
 0<\eps\leq\min\{R,\eps_*\},\qquad \eps<\min_Dw
\]
when necessary. We can also impose all the finitely many radii
required for the chartwise local algebra theorem. The two-step map
then lands in $\mathcal C_2^{\mathrm{str}}$ on all of $S(\eps)$.
This is the only shrinking used for straightness.

For $x\in S(\eps/q^m)$, the height identity gives
$\Phi^i(x)\in S(\eps/q^{m-i})$. Every consecutive two-arrow
subchain of
\[
 C_m(x)=(c(x),c(\Phi(x)),\ldots,c(\Phi^{m-1}(x)))
\]
therefore satisfies the fixed two-step condition. Lemma
\ref{lem:straight-chains} gives the analytic factorization
$j_m=\iota_m^{-1}C_m$. Its endpoints are $x$ and $\Phi^m(x)$.
Dropping the first arrow gives $C_m(\Phi(x))$, so that
$\rho_mj_{m+1}=j_m\Phi$. Set $j_0=\mathrm{id}$.
\end{proof}

\begin{proposition}[Anticanonical components and stabilization]
\label{prop:stabilization}
With the radius and maps of Theorem~\ref{thm:tower}, put
\[
 Y_m=j_m(S(\eps/q^m))\subset\cH_m.
\]
Then $j_m$ is an open immersion, and $Y_m$ is open and closed in
$t_m^{-1}(S(\eps))$. For $m\geq1$,
\begin{equation}\label{eq:cartesian}
 Y_{m+1}=\rho_m^{-1}(Y_m).
\end{equation}
In particular, at full level the condition of belonging to the
anticanonical tower is an open condition imposed already at level $K_1$.
\end{proposition}
\begin{proof}
The identity $s_mj_m=\mathrm{id}$ makes $j_m$ a section of a
finite etale map over its source domain, hence an open immersion.
Viewed over $S(\eps)$, it is a map from a finite space to the
finite space $t_m^{-1}(S(\eps))$, and is therefore finite.
Its image is consequently open and closed in that space.

Compatibility gives $Y_{m+1}\subset\rho_m^{-1}(Y_m)$ as an
open and closed subspace. Both sides are finite etale over $Y_m$.
The left side has degree $q^n$ by the identification with $\Phi$;
the right side has degree $[K_m:K_{m+1}]=q^n$. Thus the inclusion
is an isomorphism. Induction gives stabilization at level $K_1$.
Without this argument, an intersection of infinitely many open
conditions need not be open.
\end{proof}
\begin{remark}[Comparison with Siegel anticanonical tower]
In the Siegel case, with \(q=p\) and the usual normalization of the
Hasse invariant, our \(Y_m\) agrees with  anticanonical locus
\(X_{\Gamma_0(p^m)}(\varepsilon)_{\mathrm a}\)
\cite[Theorem III.2.15 and Remark III.2.16]{Scholze2015}.
Indeed, let \(A\) correspond to a point of \(S(\varepsilon/p^m)\), and
let \(C_m\subset A[p^m]\) be its canonical subgroup. The point \(j_m(A)\)
of the level-\(m\) correspondence is, in Scholze's notation,
\[
       \bigl(A/C_m,\; A[p^m]/C_m\bigr).
\]
With our convention for the two projections, \(s_mj_m(A)=A\), whereas
\(t_mj_m(A)=A/C_m=\Phi^m(A)\). Thus the level structure on the
\emph{target} is the subgroup \(D_m=A[p^m]/C_m\), which is
anticanonical: its \(p\)-torsion is disjoint from the weak canonical
subgroup of \(A/C_m\). The compatibility
\(\rho_mj_{m+1}=j_m\Phi\) is the corresponding compatibility of these
subgroups under successive quotients. In particular,
\(Y_{m+1}=\rho_m^{-1}(Y_m)\) is the analogue of Scholze's Cartesian
description of the anticanonical tower.
\end{remark}
\begin{theorem}[Perfectoid anticanonical limit]\label{thm:anti}
The diamond
\[
 Y_\infty=\ilim_m(Y_m)_C^\diamond
\]
is represented by a perfectoid space.
\end{theorem}
\begin{proof}
These are affinoid perfectoid limits of
the required towers, by Theorem~\ref{thm:tower}(2),(3).
Their overlaps are open subdiamonds because the charts are inverse
images of opens at finite level. The universal property of the
inverse-limit diamonds identifies the overlaps and their cocycles.
They therefore glue to a perfectoid space representing $Y_\infty$.
\end{proof}

\section{Full level and the perfectoid Hasse neighborhood}
\label{sec:global}

Define the full infinite-level object from the outset as the diamond
\[
 S_{\infty,C}=\ilim_{K'\subset K}(\Sh_{K^pK',C}^{\an})^\diamond.
\]
We do not assume that this inverse limit is an adic space. Let
$W\subset S_{\infty,C}$ be the inverse image of $(Y_1)_C$.
Proposition~\ref{prop:stabilization} makes $W$ an open subdiamond
and identifies it with the full-level tower over $Y_\infty$.

\begin{proposition}[Passage to full level]\label{prop:full}
The space $W$
is represented by a perfectoid space.
\end{proposition}
\begin{proof}
For $r\geq1$ put $N_r=K(p^r)$ and
$L_{m,r}=K_m\cap N_r$. Let $Z_{m,r}$ be the inverse image of
$Y_m$ at level $L_{m,r}$. For fixed $r$, the decreasing images of
$K_m$ in the finite group $K/N_r$ stabilize. Hence for all sufficiently
large $m$,
\[
 K_m=K_{m+1}(K_m\cap N_r).
\]
The corresponding squares of level covers are cartesian. Equivalently,
the natural map to the fiber product is a finite etale map inducing a
bijection on level-structure fibers; the displayed group equality
is precisely the required transitivity. Thus
\[
 Z_r=\ilim_m(Z_{m,r})_C^\diamond
\]
is the pullback of one finite etale cover to $Y_\infty$.
Almost purity makes $Z_r$ perfectoid.

The maps $Z_{r+1}\to Z_r$ are finite etale and surjective. Over an
affinoid perfectoid open in $Y_\infty$, all these covers are affinoid
perfectoid. The pullback maps on their algebras are isometric for the
spectral norms, by surjectivity. Complete their union for this norm.
The resulting norm is power-multiplicative, so its unit ball is the
power-bounded subring and the completed algebra is uniform. Every
element of the unit ball is congruent modulo $p$ to an element of a
finite-stage unit ball. Frobenius surjectivity at that stage then
proves Frobenius surjectivity modulo $p$ in the completion. This gives
an affinoid perfectoid limit. Choose the plus subring by taking the
integral closure of the closure of the union of the finite-stage plus
subrings. The inverse-limit identification is the usual one
of~\cite[Propositions 2.4.2 and 2.4.5]{ScholzeWeinstein2013}.
Finally, the subgroups $L_{m,r}$ are cofinal among the compact open
subgroups defining full level as $r$ increases. Equation
\eqref{eq:cartesian} identifies the resulting open subdiamond with
$W$. Hence $W=\ilim_r Z_r$ is perfectoid.
\end{proof}

\begin{proposition}[Saturation of the anticanonical open]
\label{prop:saturation}
Let $\operatorname{pr}_K:S_{\infty,C}\to S_C^\diamond$ forget the
$p$-level structure, and put
\[
 \mathcal N_\eps=\operatorname{pr}_K^{-1}(S(\eps)_C^\diamond).
\]
Then $\mathcal N_\eps$ is a quasicompact open perfectoid subspace of
$S_{\infty,C}$. More precisely,
\begin{equation}\label{eq:K-saturation}
 \mathcal N_\eps=\bigcup_{k\in K}kW,
\end{equation}
and finitely many translates suffice.
\end{proposition}
\begin{proof}
The map $Y_1\to S(\eps)$ is finite etale and surjective, by
Theorem~\ref{thm:tower}(1),(3). On every geometric fiber of the
$K$-torsor $S_{\infty,C}\to S_C^\diamond$, some choice of level
structure therefore lies in $W$. Transitivity of $K$ on that fiber
proves~\eqref{eq:K-saturation}. Since $W$ is the inverse image of
$Y_1$ at level $K_1$, it is $K_1$-stable. There are at most
$[K:K_1]$ distinct translates in this union. They are perfectoid by
Proposition~\ref{prop:full}; their open overlaps glue as open
subdiamonds. Finally $S(\eps)$ is quasicompact, and the projection
from the inverse limit of finite covers is quasicompact. Hence
$\mathcal N_\eps$ is quasicompact.
\end{proof}

\section{The period neighborhood and rational Hecke covering}
\label{sec:period}

We work with the diamond $X=S_{\infty,C}$ throughout the argument,
without assuming that it is represented by an adic space. Since the
finite-level Shimura varieties are proper, their analytifications are
quasicompact and quasiseparated. The finite transition maps and
\cite[Lemma 11.22]{ScholzeDiamonds} show that $X$ is a spatial diamond
and that its topological space is the inverse limit of the
finite-level topological spaces. In particular $|X|$ is spectral
and compact for the constructible topology.

\begin{proposition}[The period map before perfectoidness]
\label{prop:period-map}
There is a $G(\mathbb Q_p)$-equivariant morphism of diamonds
\[
 \pi_\HT:X\longrightarrow\Fl_C^\diamond
\]
classifying the Hodge--Tate modification of the universal aperture
with its rational etale framing. Its underlying map is spectral.
\end{proposition}
\begin{proof}
The universal aperture realizes the canonical crystalline, hence de
Rham, $K$-local system. The construction of
\cite[Proposition 2.6.3 and \S2.6.2, especially (2.6.6)]{PR2024}
associates to a trivialization of that local system its
$B_{\mathrm{dR}}^+$-lattice. For a minuscule conjugacy class the
corresponding Schubert cell is the flag variety. This gives the
map on perfectoid test objects and hence on the full-level diamond.
The construction is tensorial and functorial for rational changes
of framing, so it is $G(\mathbb Q_p)$-equivariant. In particular,
the construction does not require representability of $X$ by a
perfectoid space; see also~\cite[Remark 2.6.7]{PR2024}.

A morphism to the separated analytic flag variety from this
quasicompact spatial diamond is quasicompact: its graph is a closed
immersion, and projection from $X\times\Fl_C^\diamond$ to the second
factor is quasicompact. Thus inverse images of quasicompact opens
are quasicompact opens. The induced continuous map is spectral.
\end{proof}

\subsection{The ordinary period orbit and contraction}

Let $\chi$ be the dominant representative of the cocharacter defining
the Hodge--Tate flag variety, with its actual filtration convention,
and write $\Fl_C=G_C/P_{\chi,C}$. Here $P_\chi$ denotes the parabolic
of nonnegative weights. Choose an unramified
maximal torus in a rational Borel and choose $\chi$ in its dominant
chamber. With $d=[E:\mathbb Q_p]$ as before, put
\[
 \lambda=\sum_{i=0}^{d-1}\sigma^i\chi\in X_*(T)^{\sigma},
 \qquad P=P_\lambda,\qquad f_0=P_\chi\in\Fl(E).
\]
The canonical torus-valued ordinary aperture, constructed by the
Lubin--Tate reflex norm, has period $f_0$ in a compatible framing;
see~\cite[Remark 3.6.19]{MY2026}.
Define the ordinary period locus
\[
 \Omega=G(\mathbb Q_p)f_0\subset|\Fl_C|.
\]
For a rational Hodge cocharacter this is the usual rational flag
locus. In general one must use this orbit, rather than write
$\Fl(\mathbb Q_p)$ for a flag variety whose type need not be rational.
For the covering argument it suffices to use this explicit orbit;
no comparison theorem for all Newton strata is needed.

The conjugates of $\chi$ are dominant in the same chamber. The root
criterion for their parabolics gives
\[
 \bigcap_{i=0}^{d-1}P_{\sigma^i\chi}=P_\lambda.
\]
A rational element fixing $f_0$ belongs to every conjugate parabolic.
Consequently
\begin{equation}\label{eq:ordinary-orbit}
 \Omega=G(\mathbb Q_p)/P(\mathbb Q_p).
\end{equation}
This is a compact space: the right side is the rational-point space
of the projective rational flag variety $G/P$ (equivalently, use
Iwasawa decomposition). Its map to $|\Fl_C|$ is continuous.
Every point of $\Omega$ is a classical $C$-point and has no proper
generizations in the analytic adic space.

\begin{lemma}[Rational translates of an ordinary flag neighborhood]
\label{lem:flag-cover}
If $U\subset\Fl_C$ is an analytic open subset with $U\cap\Omega\ne\varnothing$,
then
\[
 \Fl_C=\bigcup_{g\in G(\mathbb Q_p)}gU.
\]
Finitely many translates suffice.
\end{lemma}
\begin{proof}
Translate $U$ so that it contains $f_0$. Let $\mathcal B$ be the
opposite big cell $U_\chi^-f_0$ in $G_C/P_{\chi,C}$. On a root
coordinate $z_\alpha$ of this cell one has
$\langle\alpha,\chi\rangle<0$. Since all conjugates of $\chi$ are
dominant, $\langle\alpha,\lambda\rangle<0$ as well. Thus the
\emph{rational} element $b=\lambda(p)^{-1}\in G(\mathbb Q_p)$ acts by
\[
 z_\alpha\longmapsto p^{e_\alpha}z_\alpha,
 \qquad e_\alpha=-\langle\alpha,\lambda\rangle>0.
\]
These coordinates can be chosen after a finite splitting extension,
inside $C$. The minuscule opposite radical is a vector group; the
same conclusion also follows using ordered root coordinates.
Every analytic neighborhood of $f_0$ contains a sufficiently small
closed polydisc in this cell. Every point of $\mathcal B^{\an}$ lies
in some bounded polydisc. It follows that for each such point $y$
one has $b^Ny\in U$ for sufficiently large $N$. This argument applies
to higher-rank analytic points as well, since the inequalities
are uniform on bounded polydiscs. Hence
$\mathcal B^{\an}\subset\bigcup_{N\geq0}b^{-N}U$.

Finally the $G(\mathbb Q_p)$-translates of $\mathcal B^{\an}$ cover
$\Fl_C$. To see this at any analytic point $y$, pass to its completed
residue field. The elements carrying $y$ into $\mathcal B$ form a
nonempty Zariski open in the base change of $G$. The rational points
$G(\mathbb Q_p)$ are Zariski dense after any field extension: a
polynomial with extended coefficients vanishing on all rational
points can be expanded in a finite linearly independent set of
coefficients over $\mathbb Q_p$, reducing to ordinary Zariski density.
The latter follows from smoothness and a $p$-adic analytic neighborhood
of the identity. Thus a rational translate moves $y$ into the cell.
The two assertions together prove the covering. Quasicompactness of
the analytic projective flag variety gives a finite subcover.
\end{proof}

\subsection{Ordinary period fibers and constructible compactness}

\begin{lemma}[Ordinary fibers over valued fields]
\label{lem:ordinary-fibers}
Let $D/C$ be an algebraically closed complete rank-one extension and
let $x\in X(D)$, with hyperspecial image extending to
$\bar x\in\cS(\mathcal O_D)$. Then
\[
 \pi_\HT(x)\in\Omega
 \quad\Longleftrightarrow\quad
 \bar x\bmod\mathfrak m_D\text{ is }\mu\text{-ordinary}.
\]
In the left side $\Omega$ is viewed in $\Fl(D)$ by base change.
For every positive $r<1$, this implies the containment on all adic
points
\begin{equation}\label{eq:ordinary-fiber-containment}
 |\pi_\HT|^{-1}(\Omega)\subset|\operatorname{pr}_K^{-1}(S(r)_C^\diamond)|.
\end{equation}
\end{lemma}
\begin{proof}
Properness gives the extension $\bar x$. Pull its aperture back to
$\mathcal O_D$ and use its prismatic realization. In every algebraic
representation of $G$, this gives a finite free Breuil--Kisin--Fargues
module, with the compatible tensor data. Its etale lattice is the
one trivialized by $x$, and its $B_{\mathrm{dR}}^+$-lattice is the
modification classified by $\pi_\HT(x)$. The comparison with this
modification follows either from the crystalline comparison or from
the shtuka realization of the prismatic crystal; see
\cite[\S\S3.8--3.9 and the construction in the proof of
Proposition 8.7.9]{MY2026} and~\cite[\S2.6]{PR2024}.

Over $\mathcal O_D$, the classification
\cite[Theorem 14.1.1, parts (2) and (5)]{ScholzeWeinstein2020}
identifies a Breuil--Kisin--Fargues module with its etale lattice and
its $B_{\mathrm{dR}}^+$-lattice. Applying the equivalence to every
representation, compatibly with tensor products, gives the analogous
identification for the $G$-object. After inverting $p$ the datum is a
rational etale framing and a $G$-modification. The flag records this
modification because it has the fixed minuscule bound. This statement
concerns the $G$-modification itself; weights of an individual
representation need not all lie in two consecutive degrees.

Let $\mathfrak Q_0$ be the canonical torus-valued ordinary aperture
with period $f_0$. If $\pi_\HT(x)=gf_0$ for $g\in G(\mathbb Q_p)$,
then $g$ identifies the rational etale spaces and their
$B_{\mathrm{dR}}^+$-lattices. The classification just recalled gives
an isomorphism of rational Breuil--Kisin--Fargues objects, respecting
all tensors, between $\mathfrak Q_{\bar x}$ and $\mathfrak Q_0$.
Specializing their prismatic realizations along
$A_{\inf}(\mathcal O_D)\to W(k_D)$ and inverting $p$ gives
isomorphic $G$-isocrystals. This is the crystalline specialization
compatible with the etale and de Rham realizations
\cite[\S\S3.8--3.9]{MY2026}; compare the vector-bundle case
\cite[Corollary 14.8.2]{ScholzeWeinstein2020}.
The isocrystal of $\mathfrak Q_0$ is the ordinary class.
By~\cite[Lemma 3.6.14]{MY2026}, this is equivalent to ordinary
reduction of the aperture, hence of $\bar x$.

For the converse, an ordinary aperture has the canonical parabolic
reduction of~\cite[Proposition 3.6.12]{MY2026}. Its Levi aperture is
etale, and at each truncation the remaining datum is a torsor under
the finite flat group of~\cite[Proposition 3.6.8]{MY2026}. Every
finite flat torsor over $\mathcal O_D$ is trivial: it has a $D$-point,
since $D$ is algebraically closed, and that point extends by properness
of a finite morphism. The Levi torsors are trivial for the same
reason. Thus each truncation is isomorphic to the corresponding
truncation of $\mathfrak Q_0$. The sets of such isomorphisms are finite
and nonempty, and form an inverse system; compactness of finite sets
gives a compatible choice. The full apertures are consequently
isomorphic. Their periods differ by the change of etale framing,
which belongs to $K$, and in particular the period of $x$ lies in
$\Omega$.

We finally justify~\eqref{eq:ordinary-fiber-containment} for higher-rank
points. Represent a point above $\Omega$ by
$\operatorname{Spa}(D,D^+)$, with $D$ algebraically closed and $D^+$
a valuation subring of $\mathcal O_D$. Restriction to
$\operatorname{Spa}(D,\mathcal O_D)$ gives its rank-one generization.
Its period is still the same point of $\Omega$, so its hyperspecial
image has ordinary reduction by the preceding argument. Choose a
formal chart containing the specialization of the original point
and a local Hasse lift $h$. The rank-one valuation of $h$ is zero.
In the refined valuation its value may differ infinitesimally from
one, but it is strictly larger than $|p|^r$ for every $r>0$.
Consequently the original point belongs to $S(r)$ as well.
We have deliberately used a positive-radius neighborhood: equality
with the ordinary tube itself need not hold at higher-rank points.
\end{proof}

\begin{lemma}[A neighborhood from constructible compactness]
\label{lem:constructible-neighborhood}
Let $f:T\to F$ be a spectral map of spectral spaces and let
$A\subset F$ be compact, with each of its points having no proper
generizations. If $V\subset T$ is open and $f^{-1}(A)\subset V$,
then there is a quasicompact open $U\subset F$ such that
\[
 A\subset U,\qquad f^{-1}(U)\subset V.
\]
\end{lemma}
\begin{proof}
The intersection of all quasicompact open neighborhoods of $A$ is
$A$. Indeed, for $z\notin A$ and each $a\in A$, the condition on
generizations gives a quasicompact open containing $a$ but not $z$.
A finite number of these cover $A$, so their union is a
quasicompact open neighborhood avoiding $z$.

The complement $B=T\setminus V$ is closed, hence compact, for the
constructible topology: every spectral open is constructibly open.
For each quasicompact open neighborhood $U$ of $A$, the set
$B\cap f^{-1}(U)$ is constructibly closed, since $f$ is spectral.
Their intersection is $B\cap f^{-1}(A)=\varnothing$. Constructible
compactness gives a finite subfamily with empty intersection.
Intersecting the corresponding neighborhoods gives the required $U$.
\end{proof}

\begin{proposition}[The ordinary flag neighborhood]
\label{prop:period-neighborhood}
For the fixed positive radius $\eps$ of Theorem~\ref{thm:tower},
there is a quasicompact analytic open $U\subset\Fl_C$ such that
\begin{equation}\label{eq:period-neighborhood}
 \Omega\subset U,\qquad \pi_\HT^{-1}(U)\subset\mathcal N_\eps.
\end{equation}
In particular $\pi_\HT^{-1}(U)$ is represented by a perfectoid space.
\end{proposition}
\begin{proof}
Apply Lemma~\ref{lem:constructible-neighborhood} to
$|\pi_\HT|:|X|\to|\Fl_C|$, $A=\Omega$, and
$V=|\mathcal N_\eps|$. Spectrality and constructible compactness were
established before Proposition~\ref{prop:period-map}, and the fiber
containment is Lemma~\ref{lem:ordinary-fibers}. The orbit $\Omega$
is compact and consists of points without proper generizations,
as observed after~\eqref{eq:ordinary-orbit}.
Quasicompact opens of the analytic flag variety define analytic
open subspaces, giving~\eqref{eq:period-neighborhood}.
Proposition~\ref{prop:saturation} makes $\mathcal N_\eps$ perfectoid;
its open subdiamond $\pi_\HT^{-1}(U)$ is therefore perfectoid as well.
\end{proof}

\begin{theorem}[Perfectoidness at full $p$-level]
\label{thm:global}
Let $\cS/\OEv$ be a smooth proper limpid scheme model, with 
neat tame level $K^p$. Then
\[
 S_{\infty,C}=\ilim_{K'\subset K}(\Sh_{K^pK',C}^{\an})^\diamond
\]
is represented by a perfectoid space over $C$. It is covered by
finitely many $G(\mathbb Q_p)$-translates of the perfectoid
anticanonical open $W$. No separate geometric-tower or Hecke-covering
hypothesis is required.
\end{theorem}
\begin{proof}
Choose $U$ as in Proposition~\ref{prop:period-neighborhood}.
Lemma~\ref{lem:flag-cover} gives a finite cover of $\Fl_C$ by
translates $g_iU$, with $g_i\in G(\mathbb Q_p)$. Equivariance gives
\[
 X=\bigcup_{i=1}^s g_i\pi_\HT^{-1}(U)
   \subset\bigcup_{i=1}^s g_i\mathcal N_\eps\subset X.
\]
The middle union is therefore $X$. Its members are open perfectoid
subdiamonds by Proposition~\ref{prop:saturation}. Their overlaps are
open subspaces of perfectoid spaces, and the given diamond
identifications supply compatible gluing maps. Gluing yields a
perfectoid space representing $X$. Finally
\eqref{eq:K-saturation} replaces each $\mathcal N_\eps$ by finitely
many $K$-translates of $W$, proving the last assertion.
\end{proof}

\bibliography{myref}{}

@book{ScholzeWeinstein2020,
  title={Berkeley Lectures on {$p$}-adic Geometry},
  author={Scholze, Peter and Weinstein, Jared},
  series={Annals of Mathematics Studies},
  volume={207},
  year={2020},
  publisher={Princeton University Press}
}

@article{PR2024,
  title={{$p$}-adic {S}htukas and the Theory of Global and Local {S}himura Varieties},
  author={Pappas, Georgios and Rapoport, Michael},
  journal={Cambridge Journal of Mathematics},
  volume={12},
  number={1},
  pages={1--164},
  year={2024}
}

@article{ScholzeWeinstein2013,
  title={Moduli of {$p$}-Divisible Groups},
  author={Scholze, Peter and Weinstein, Jared},
  journal={Cambridge Journal of Mathematics},
  volume={1},
  number={2},
  pages={145--237},
  year={2013},
  publisher={International Press of Boston}
}

@misc{ScholzeDiamonds,
  title={{\'E}tale Cohomology of Diamonds},
  author={Scholze, Peter},
  year={2017},
  eprint={1709.07343},
  archivePrefix={arXiv}
}

@article{he2026perfectoidness,
  title={Perfectoidness via Sen theory and applications to Shimura varieties},
  author={He, Tongmu},
  journal={Journal of the American Mathematical Society},
  volume={39},
  number={1},
  pages={95--176},
  year={2026}
}

@misc{bhatt2022prismatic,
  title={Prismatic {$F$}-Gauges},
  author={Bhatt, Bhargav},
  year={2022},
  note={Lecture notes for MAT 549, Princeton University},
  url={https://www.math.ias.edu/~bhatt/teaching/mat549f22/lectures.pdf}
}

@book{BH1999,
  title={Metric Spaces of Non-Positive Curvature},
  author={Bridson, Martin R. and Haefliger, Andr{\'e}},
  series={Grundlehren der mathematischen Wissenschaften},
  volume={319},
  year={1999},
  publisher={Springer}
}

@misc{KW2018,
  title={Generalized {$\mu$}-ordinary {H}asse Invariants},
  author={Koskivirta, Jean-Stefan and Wedhorn, Torsten},
  year={2018},
  eprint={1406.2178},
  archivePrefix={arXiv}
}

@misc{MY2026,
  title={On Canonicity for Integral Models of {S}himura Varieties with Hyperspecial Level},
  author={Madapusi, Keerthi and Youcis, Alex},
  year={2026},
  note={Version 1},
  eprint={2604.06442},
  archivePrefix={arXiv}
}

@article{Scholze2012,
  title={Perfectoid Spaces},
  author={Scholze, Peter},
  journal={Publications Math{\'e}matiques de l'IH{\'E}S},
  volume={116},
  number={1},
  pages={245--313},
  year={2012},
  publisher={Springer}
}

@article{Scholze2015,
  title={On Torsion in the Cohomology of Locally Symmetric Varieties},
  author={Scholze, Peter},
  journal={Annals of Mathematics},
  volume={182},
  number={3},
  pages={945--1066},
  year={2015},
  publisher={Princeton University}
}

@incollection{Yu2009,
  title={{B}ruhat--{T}its Theory and Buildings},
  author={Yu, Jiu-Kang},
  booktitle={{O}ttawa Lectures on Admissible Representations of Reductive {$p$}-adic Groups},
  series={Fields Institute Monographs},
  volume={26},
  pages={53--77},
  year={2009},
  publisher={American Mathematical Society}
}
\bibliographystyle{alphaurl}

\end{document}